\documentclass[12pt]{article}

\usepackage{amsmath,amssymb,amsthm,extsizes}

\newtheorem*{theorem}{Theorem}
\newtheorem*{lemma}{Lemma}

\usepackage[active]{srcltx}
\title{A holomorphic
transformation to a miniversal deformation
under *congruence does not always exist}
\author{Lena Klimenko (e.n.klimenko@gmail.com)\\National
Transport University,
 Suvorov 1, Kiev, Ukraine
}

\date{}

\begin{document}
\maketitle

\begin{abstract}
V.I. Arnold [Russian Math. Surveys 26
(2) (1971) 29--43] constructed
miniversal deformations of square
complex matrices under similarity.
Reduction transformations to them and
also to miniversal deformations of
matrix pencils and matrices under
congruence can be taken holomorphic. We
prove that this is not true for
reduction transformations to miniversal
deformations of matrices under
*congruence.

{\it AMS classification:} 15A21, 15A63,
47A07

{\it Keywords:} Sesquilinear forms;
*Congruence; Miniversal deformations
\end{abstract}

The reduction of a matrix to its Jordan
form is an unstable operation: both the
Jordan form and the reduction
transformation depend discontinuously
on the entries of the original matrix.
Therefore, if the entries of a matrix
are known only approximately, then it
is unwise to reduce it to Jordan form.
Furthermore, when investigating a
family of matrices smoothly depending
on parameters, then although each
individual matrix can be reduced to a
Jordan form, it is unwise to do so
since in such an operation the
smoothness relative to the parameters
is lost.

For these reasons, Arnold \cite{arn}
constructed miniversal deformations of
matrices under similarity; that is, a
simple normal form to which not only a
given square matrix $A$ but all
matrices $B$ close to it can be reduced
by similarity transformations that
smoothly depend on the entries of $B$.
Miniversal deformations were also
constructed for matrix pencils
\cite{kag,gar_ser,k-s_triang}, matrices
under congruence \cite{f_s}, and
matrices under *congruence
\cite{def-sesq} (two complex matrices
$A$ and $B$ are \emph{*congruent} if
$A=S^*BS$ for some nonsingular $S$).
For all matrices in a neighborhood of a
given matrix with respect to similarity
or congruence and all matrix pencils in
a neighborhood of a given pencil with
respect to equivalence, reduction
transformations to miniversal
deformations can be taken holomorphic.
We prove that reduction transformations
to miniversal deformations under
*congruence cannot be taken holomorphic
even if we restrict ourselves to
$1\times 1$ matrices. All matrices that
we consider are complex matrices.

Let $[a]$ be any nonzero $1\times 1$
matrix; let $a=re^{i\varphi }$ with
$r>0$. Then $[a]$ is *congruent via
$S=[\sqrt r]$ to $[b]:=[e^{i\varphi
}]$, which is a canonical form of $[a]$
for *congruence. All matrices
$[a+\varepsilon ]$ that are
sufficiently close to $[a]$ can be
simultaneously reduced by some
transformation
\begin{equation}\label{tef}
[s(\varepsilon)]^* [a+\varepsilon ]
[s(\varepsilon)],\qquad
\text{$s(\varepsilon )$
is continuous, }
s(0)=\sqrt r,
\end{equation}
to the form
\begin{equation}\label{fru}
[\varphi(\varepsilon )]= \left\{
  \begin{array}{ll}
    [b+\alpha(\varepsilon)] &
\hbox{if $a\notin \mathbb R$,} \\ {}
   [ b+\alpha(\varepsilon)i] &
\hbox{if $a\in \mathbb R$,}
  \end{array}
\right.\qquad\text{where
$\alpha(\varepsilon)$ is real valued},
\end{equation}
which is the miniversal deformation of
$[a]$ under *congruence (see
\cite{def-sesq}).

The purpose of this note is to show
that the complex functions
$s(\varepsilon )$ and $\varphi
(\varepsilon )$ in \eqref{tef} and
\eqref{fru} cannot be taken holomorphic
at zero.

Recall that if a complex-valued
function $f(z)$ is holomorphic at $0$,
then it is holomorphic in some
neighborhood $U$ of $0$ and the
Cauchy--Riemann equations
\begin{equation}\label{juw}
 \frac{\partial u}{\partial x}(x_0,y_0)
=\frac{\partial v}{\partial y}(x_0,y_0)
,\qquad
\frac{\partial u}{\partial y}(x_0,y_0)=
\frac{\partial v}{\partial x}(x_0,y_0)
\end{equation}
hold for all $x_0+iy_0\in U$, where
$u(x,y)$ and $v(x,y)$ are the real and
imaginary parts of $f(z)$:
\[
f(x+iy)=u(x,y)+iv(x,y),\qquad
x,y,u(x,y),v(x,y)\in\mathbb R.
\]

We can suppose that $a=b$. For
simplicity we suppose that $b=1$.

\begin{theorem}
If $a=b=1$, then the complex functions
$s(\varepsilon )$ and $\varphi
(\varepsilon )=1+\alpha(\varepsilon)i$
in \eqref{tef} and \eqref{fru} cannot
be taken holomorphic at zero.
\end{theorem}

Represent $\varepsilon$ in the form
$\varepsilon =-1+x+iy$, in which
$(x,y)\in\mathbb R^2$ is in a
neighborhood of $(1,0)$. Then
$[a+\varepsilon] =[x+iy]$ is reduced by
some transformation \eqref{tef} of the
form
\begin{equation*}\label{fwi}
 [x+iy]\mapsto [|s(\varepsilon )|^2(x+iy)]
\end{equation*}
to the miniversal deformation
$[1+\alpha(\varepsilon)i]$ with
$\alpha(\varepsilon)
 \in\mathbb R$.
Therefore, $|s(\varepsilon
)|^2(x+iy)=1+\alpha(\varepsilon)i$, and
so
\begin{equation}\label{flm}
 |s(\varepsilon )|^2x=1,\qquad
|s(\varepsilon
)|^2y=\alpha(\varepsilon).
\end{equation}
By these equalities,
\[
\varphi
(\varepsilon )=1+\alpha(\varepsilon)i=
1+|s(\varepsilon
)|^2yi=1+\frac yx i.
\]
The real part of $\varphi (\varepsilon
)$ is $1$, the imaginary part is $y/x$,
they do not satisfy \eqref{juw}, and so
$\varphi (\varepsilon )$ is not
holomorphic.

Note that the holomorphicity of
$[s(\varepsilon )]$ does not imply the
holomorphicity of $[s(\varepsilon
)]^*$, and so the non-holomorphicity of
$\varphi (\varepsilon )$ does not imply
the non-holomorphicity of
$s(\varepsilon )$.

The first equality in \eqref{flm} is
represented in the form
\begin{equation}\label{dey}
 u(x,y)^2+v(x,y)^2=\frac 1x
\end{equation}
in which $u(x,y)$ and $v(x,y)$ are the
real and imaginary parts of
$s(\varepsilon)$.

The function $s(\varepsilon )$ cannot
be taken holomorphic at $0$ due to the
following lemma.

\begin{lemma}
There exist no real functions $u(x,y)$
and $v(x,y)$ in a neighborhood of
$(1,0)$ such that \eqref{dey} holds and
$u(x,y)+iv(x,y)$ is holomorphic.
\end{lemma}

\begin{proof}
To the contrary, let such $u(x,y)$ and
$v(x,y)$ exist. Then they must satisfy
the Cauchy--Riemann equations
\begin{equation}\label{lod}
 u'_x=v_y',\qquad u'_y=-v_x'\,.
\end{equation}
By \eqref{dey},
\begin{equation}\label{fyk}
 uu'_x+vv'_x=-\frac 1{2x^2},\qquad
 uu'_y+vv'_y=0\,.
\end{equation}
\medskip

\emph{Step 1.} Substituting \eqref{lod}
into the second equation of
\eqref{fyk}, we obtain the following
system of linear equations with respect
to $u'_x$ and $v'_x$:
\begin{equation}\label{hku}
\begin{aligned}
uu'_x+vv'_x&=-\frac 1{2x^2}\\
-vu'_x+uv'_x&=0
\end{aligned}
\end{equation}
Its determinant
is \eqref{dey}. By Cramer's rule,
\begin{equation}\label{jus}
u'_x=\frac{\begin{vmatrix}
-\frac 1{2x^2}&v\\0&u
\end{vmatrix}}{\frac 1x}=-\frac u{2x},
\qquad
v'_x=\frac{\begin{vmatrix}
u&-\frac 1{2x^2}\\-v&0
\end{vmatrix}}{\frac 1x}=-\frac v{2x}\,.
\end{equation}
From the first equation, $
\frac{\partial u}{\partial x}=-\frac
u{2x}$, thus $\frac{\partial u}u=-\frac
{\partial x}{2x}$,
\[
\ln u=-\frac 12\ln x+\ln C(y)=\ln x^{
-\frac 12}C(y),\qquad u=C(y)x^{
-\frac 12}\,.
\]
From the second equation in
\eqref{jus}, $v=D(y)x^{ -\frac 12}$. We
have that all solutions of the system
\eqref{hku} are given by
\begin{equation}\label{geo}
u=\frac {C(y)}{\sqrt x},\qquad
v=\frac {D(y)}{\sqrt x}
\end{equation}
(see \cite[Chapter IV]{hart}).

\medskip

\emph{Step 2.} Substituting \eqref{lod}
into the first equation of \eqref{fyk},
we obtain the following system of
linear equations with respect to $u'_y$
and $v'_y$:
\begin{align*}
vu'_y-uv'_y&=\frac 1{2x^2}\\
uu'_y+vv'_y&=0
\end{align*}
Its determinant is \eqref{dey}. By
Cramer's rule,
\begin{equation*}\label{jus1}
u'_y=\frac{\begin{vmatrix}
\frac 1{2x^2}&-u\\0&v
\end{vmatrix}}{\frac 1x}=\frac v{2x},
\qquad
v'_y=\frac{\begin{vmatrix}
v&\frac 1{2x^2}\\u&0
\end{vmatrix}}{\frac 1x}=-\frac u{2x}\,.
\end{equation*}
We obtain
\[
u''_{yy}=\frac {v'_y}{2x}=-\frac u{4x^2},\qquad
v''_{yy}=-\frac {u'_y}{2x}=-\frac u{4x^2}\,.
\]
The general solutions of
$u''_{yy}+\frac 1{4x^2} u =0$ and
$v''_{yy}+\frac 1{4x^2} v =0$ are
\begin{equation}\label{fhx}
 u=A(x)\cos \frac
y{2x}+B(x)\sin\frac
y{2x},\quad v=A_1(x)\cos \frac
y{2x}+B_1(x)\sin\frac
y{2x}\,.
\end{equation}

By \eqref{geo}, the functions $u\sqrt x
$ and $v\sqrt x $ do not depend on $x$.
By \eqref{fhx}, they do not depend of
$x$ only if $u$ and $v$ are identically
equal to $0$, which contradicts to
\eqref{dey}.
\end{proof}

\end{document}